\documentclass[12pt]{article}
\usepackage[pdftex]{graphicx,color}
\usepackage{amssymb}
\usepackage{amsmath}

\newcommand\PG{{\rm PG}}
\newcommand\PGL{{\rm PGL}}

\newcommand\AG{{\rm\mbox{AG}}}
\newcommand\GF{{\rm\mbox{GF}}}

\renewcommand{\P}{\mathcal{P}}

\newcommand\bs{\backslash}
\newcommand\st{:}
\newcommand\A{{\cal A}}

\newcommand{\Label}{\label}

\newtheorem{theorem}{Theorem}[section]
\newtheorem{lemma}[theorem]{Lemma}
\newtheorem{corollary}[theorem]{Corollary}

\newenvironment{proof}{\noindent{\bf Proof}\hspace{0.5em}}
    { \null  \hfill $\square$ \par}

\usepackage[mathscr]{eucal}

\newcommand{\bbb}{[\Bpi]}
\newcommand{\Bpi}{{\mathscr B}}

\newcommand{\ST}{{\mathscr S}_T}
\newcommand{\SdT}{{\mathscr S}'_T}
 \newcommand{\Sl}{{\mathscr S}(\ell)}
  \newcommand{\Sm}{{\mathscr S}(m)}

\newcommand{\bbc}{{\mathcal C}}

\newcommand{\bbn}{{\mathcal N}}

\renewcommand{\r}{{q}}
\newcommand\w{\tau}

\renewcommand{\S}{{\cal S}}
\newcommand{\C}{{\cal C}}
\newcommand{\N}{{\cal N}}

\newcommand{\si}{\Sigma_\infty}
\newcommand{\li}{\ell_\infty}
\newcommand{\abb}{{\cal A(\cal S)}}
\newcommand{\pbb}{{\cal P(\cal S)}}

\newcommand{\orsps}{order-$\r$-subplanes}
\newcommand{\orsls}{order-$\r$-sublines}
\newcommand{\orsp}{order-$\r$-subplane}
\newcommand{\orsl}{order-$\r$-subline}

\newcommand{\takeaway}{\backslash}
\renewcommand{\st}{\,|\,}
\newcommand{\Xl}{\mathcal X}

\newcommand{\CapitalTheta}{\Theta}
\newcommand{\plustheta}{\theta}
\newcommand{\minustheta}{\theta^{\mbox{\rm {-}}}\;\!\!\!}
\newcommand{\ee}{\frac{\minustheta(e)}{\theta(e)}}
\newcommand{\eeq}{\frac{\minustheta(e)^\r}{\theta(e)}}

\begin{document}

\title{An investigation of the tangent splash of a subplane of $\PG(2,q^3)$}
\date{}

\author{S.G. Barwick and Wen-Ai Jackson
\\ School of Mathematics, University of Adelaide\\
Adelaide 5005, Australia
}


\maketitle


Corresponding Author: Dr Susan Barwick, University of Adelaide, Adelaide
5005, Australia. Phone: +61 8 8313 3983, Fax: +61 8 8313 3696, email: susan.barwick@adelaide.edu.au

Keywords: Bruck-Bose representation, cubic extension,  subplanes, Sherk
surfaces, cover planes, linear sets.

AMS code: 51E20

\begin{abstract}
In $\PG(2,q^3)$, let $\pi$ be a subplane of order $q$ that is tangent to
$\li$. The tangent splash of $\pi$ is defined to be the set of $q^2+1$
points on $\li$ that lie on a line of $\pi$. 
This article investigates properties of the tangent splash. 
We show that all tangent splashes are projectively equivalent,  investigate sublines contained in a tangent splash, and consider the structure of a tangent splash in 
 the
Bruck-Bose representation of $\PG(2,\r^3)$ in
$\PG(6,\r)$. We 
show that a tangent splash of $\PG(1,q^3)$ is a $\GF(q)$-linear set of rank 3 and size $q^2+1$; this allows us to use results about linear sets from \cite{lavr10} to obtain properties of tangent splashes.
\end{abstract}

\section{Introduction}

In \cite{barw12}, the authors study the Bruck-Bose representation of
$\PG(2,q^3)$ in $\PG(6,q)$ and determine the representation of \orsps\  and
\orsls\ of $\PG(2,q^3)$ in $\PG(6,q)$. 
In this article we investigate in more detail the {\em tangent order-$q$-subplanes} of
$\PG(2,q^3)$, that is, the subplanes of $\PG(2,q^3)$ that have order $q$ and
 meet $\li$ in exactly one point.


Let $\pi$ be a subplane of $\PG(2,q^3)$ of order $q$ that meets $\li$ in 0 or 1 point.
Each line of $\pi$, when extended to $\PG(2,q^3)$, meets $\li$ in a
point. The set of points of $\li$ that lie on a line of $\pi$ is called the
{\em splash} of $\pi$. 
If $\pi$ is tangent to $\li$ at a point $T$, then the splash of $\pi$ consists of 
$T=\pi\cap\li$, and $q^2$ further points. We call this splash a {\em
  tangent splash} $\ST$ with {\em centre} $T$.
If $\pi$ is disjoint from $\li$, then the splash contains $q^2+q+1$ points,
and we call it an {\em exterior splash}. 
In this article we investigate the tangent splash of a tangent \orsp.
In future work the authors investigate the exterior splash.

Section~\ref{sect:intro} contains the necessary background information. Section~\ref{sect:sublines} investigates \orsls\ contained
in a tangent splash (this investigation is completed in
Section~\ref{sec:new-lin-set}). We show that the unique \orsl\ of
$\li$ 
containing the centre $T$ and two further points of a tangent splash $\ST$
is contained in $\ST$, and further, these are the only \orsls\ contained in
$\ST$. 
In Section~\ref{sect:group} we prove that the collineation group $\PGL(3,q^3)$ is transitive on
the tangent splashes of $\li$, and prove further useful group theoretic
results about tangent splashes. 
Section~\ref{sect:count} looks at counting the number of tangent splashes
on $\li$. Further, we look at how many points are needed to define a unique
\orsp\ with a given tangent splash. 

To continue with our investigation, we need
to introduce coordinates, and in Section~\ref{coord-subplane} we
coordinatise an \orsp\ of $\PG(2,q^3)$. 
This coordinatisation is used in Section~\ref{sec:new-lin-set}
 to show that a tangent splash of $\PG(1,q^3)$ is a $\GF(q)$-linear set of rank 3 and size $q^2+1$. This equivalence means that we can interpret results about linear sets from \cite{lavr10} to obtain results about tangent splashes. We can conclude that a tangent splash contains exactly $q^2+q$ \orsls, and these all contain the centre.
We  consider a tangent splash in the Bruck-Bose
representation in $\PG(6,q)$. A tangent splash corresponds to a set of $q^2+1$ planes
contained in a regular 2-spread of $\PG(5,q)$. As a consequence of the linear set equivalence,  a tangent splash in $\PG(5,q)$  has a set of $q^2+q+1$
{\em cover planes}, these meet the centre in distinct lines, and meet  every other plane  of the tangent splash in distinct points. 

In Section~\ref{sect:sherk} we
prove that a tangent splash is a Sherk surface as defined in \cite{sher86}.
In Section~\ref{sect:tgt-subspace} we begin with a tangent \orsp\ $\pi$,
and for each affine point $P\in\pi$ we construct another \orsp\ $P^\perp$,
called the tangent subspace. We then investigate the relationship between $\pi$ and $P^\perp$ in
$\PG(6,q)$.

\section[Bruck Bose representation]{The Bruck-Bose representation of $\PG(2,q^3)$ in $\PG(6,q)$}\Label{sect:intro}


\subsection[Bruck-Bose]{The Bruck-Bose representation}\Label{BBintro}

We begin by describing the 
Bruck-Bose representation of $\PG(2,q^3)$ in $\PG(6,q)$, and introduce the
notation we will use. 

A 2-{\em spread} of $\PG(5,\r)$ is a set of $\r^3+1$ planes that partition
$\PG(5,\r)$. 
A 2-{\em regulus} of $\PG(5,\r)$ is a
set of $\r+1$ mutually disjoint planes $\pi_1,\ldots,\pi_{\r+1}$ with
the property that if a line meets three of the planes, then it meets all
${\r+1}$ of them. Three mutually disjoint planes in $\PG(5,\r)$ lie on a unique
2-regulus. A 2-spread $\S$ is {\em regular} if for any three planes in $\S$, the
2-regulus containing them is contained in $\S$. 
See \cite{hirs91} for
more information on 2-spreads.

The following construction of a regular 2-spread of $\PG(5,\r)$ will be
needed. Embed $\PG(5,\r)$ in $\PG(5,\r^3)$ and let $g$ be a line of
$\PG(5,\r^3)$ disjoint from $\PG(5,\r)$. Let $g^\r$, $g^{\r^2}$ be the
conjugate lines of $g$; both of these are disjoint from $\PG(5,\r)$. Let $P_i$ be
a point on $g$; then the plane $\langle P_i,P_i^\r,P_i^{\r^2}\rangle$ meets
$\PG(5,\r)$ in a plane. As $P_i$ ranges over all the points of  $g$, we get
$\r^3+1$ planes of $\PG(5,\r)$ that partition $\PG(5,\r)$. These planes form a
regular 2-spread $\S$ of $\PG(5,\r)$. The lines $g$, $g^\r$, $g^{\r^2}$ are called the (conjugate
skew) {\em transversal lines} of the 2-spread $\S$. Conversely, given a regular 2-spread
in $\PG(5,\r)$,
there is a unique set of three (conjugate skew) transversal lines in $\PG(5,\r^3)$ that generate
$\S$ in this way.

We will use the linear representation of a finite
translation plane $\P$ of dimension at most three over its kernel,
due independently to
Andr\'{e}~\cite{andr54} and Bruck and Bose
\cite{bruc64,bruc66}. 
Let $\si$ be a hyperplane of $\PG(6,\r)$ and let $\S$ be a 2-spread
of $\si$. We use the phrase {\em a subspace of $\PG(6,\r)\takeaway\si$} to
  mean a subspace of $\PG(6,\r)$ that is not contained in $\si$.  Consider the following incidence
structure:
the \emph{points} of $\abb$ are the points of $\PG(6,\r)\takeaway\si$; the \emph{lines} of $\abb$ are the 3-spaces of $\PG(6,\r)\takeaway\si$ that contain
  an element of $\S$; and \emph{incidence} in $\abb$ is induced by incidence in
  $\PG(6,\r)$.
Then the incidence structure $\abb$ is an affine plane of order $\r^3$. We
can complete $\abb$ to a projective plane $\pbb$; the points on the line at
infinity $\li$ have a natural correspondence to the elements of the 2-spread $\S$.
The projective plane $\pbb$ is the Desarguesian plane $\PG(2,\r^3)$ if and
only if $\S$ is a regular 2-spread of $\si\cong\PG(5,\r)$ (see \cite{bruc69}).

We use the following notation: if $P$ is an affine point of $\PG(3,q^3)$, we also
use $P$ to refer to the corresponding affine point in $\PG(6,q)$. If $T$
is a point of $\li$ in $\PG(2,q^3)$, we use $[T]$ to refer to the spread
element of $\S$ in $\PG(6,q)$ corresponding to $T$. More generally, if $X$
is a set of points of $\PG(2,q^3)$, then we let $[X]$ denote the
corresponding set  in $\PG(6,q)$. 

In the case $\pbb\cong\PG(2,\r^3)$, 
we can relate the coordinates of $\PG(2,\r^3)$ and $\PG(6,\r)$ as
follows. Let $\tau$ be a primitive element in $\GF(\r^3)$ with primitive
polynomial $$x^3-t_2x^2-t_1x-t_0.$$ Then every element $\alpha\in\GF(\r^3)$
can be uniquely written as $\alpha=a_0+a_1\tau+a_2\tau^2$ with
$a_0,a_1,a_2\in\GF(\r)$. Points in $\PG(2,\r^3)$ have homogeneous coordinates
$(x,y,z)$ with $x,y,z\in\GF(\r^3)$. Let the line at infinity $\li$ have
equation $z=0$; so the affine points of $\PG(2,\r^3)$ have coordinates
$(x,y,1)$. Points in $\PG(6,\r)$ have homogeneous coordinates
$(x_0,x_1,x_2,y_0,y_1,y_2,z)$ with $x_0,x_1,x_2,y_0,y_1,y_2,z\in\GF(\r)$.  Let $\si$ have equation $z=0$. 
Let $P=(\alpha,\beta,1)$ be a point of $\PG(2,\r^3)$. We can write $\alpha=a_0+a_1\tau+a_2\tau^2$ and
$\beta=b_0+b_1\tau+b_2\tau^2$ with $a_0,a_1,a_2,b_0,b_1,b_2\in\GF(\r)$. Then
the map 
\begin{eqnarray*}
\sigma\colon \PG(2,\r^3)\takeaway\li&\rightarrow&\PG(6,\r)\takeaway\Sigma_\infty\\
(\alpha,\beta,1)&\mapsto&(a_0,a_1,a_2,b_0,b_1,b_2,1)
\end{eqnarray*}
 is the Bruck-Bose map. 
More generally, if $z\in\GF(q)$, then we can generalise this to
$$\sigma(\alpha,\beta,z)=(a_0,a_1,a_2,b_0,b_1,b_2,z).$$
Note that if $z=0$, then $T=(\alpha,\beta,0)$ is a point of $\li$, and
$\sigma(\alpha,\beta,0)$ is a single point in the spread element $[T]$
corresponding to $T$. 


\subsection[Tangent subplanes]{Subplanes and sublines in the Bruck-Bose representation}

An \emph{\orsl}\ of $\PG(1,q^3)$ is defined to be  
one of the images of $\PG(1,\r)=\{(a,1)\st a\in\GF(\r)\}\cup\{(1,0)\}$
under $\PGL(2,\r^3)$.  An \emph{\orsp} of $\PG(2,\r^3)$ is a subplane of
$\PG(2,q^3)$ of order $q$. Equivalently, it is an image of $\PG(2,\r)$ under $\PGL(3,\r^3)$. An \emph{\orsl} of
$\PG(2,q^3)$ is a line of an \orsp\ of $\PG(2,q^3)$. 

In \cite{barw12,barw13}, the authors determine the representation of \orsps\  and
\orsls\ of $\PG(2,q^3)$ in the Bruck-Bose representation in $\PG(6,q)$, and
we quote the results we need here.
 We first introduce some notation to simplify the statements. A {\em
   special conic} $\C$ is a non-degenerate conic in a spread element (a plane), such
that 
when we extend $\C$ to $\PG(6,q^3)$, it
 meets the transversals of
the regular 2-spread $\S$. Similarly, a {\em special twisted cubic} $\N$ is a
twisted cubic in a 3-space of $\PG(6,q)\backslash\si$ about a spread
element, such that when we extend $\N$ to $\PG(6,q^3)$, it meets the transversals of $\S$. 
Note that a special twisted cubic has no points in $\si$. 

\begin{theorem}{\rm \cite{barw12}}\Label{sublinesinBB}
Let $b$ be an \orsl\ of $\PG(2,q^3)$. 
\begin{enumerate}
\item\Label{subline-secant-linfty} If $b\subset\li$, then in $\PG(6,q)$, $b$ corresponds to a 
$2$-regulus of $\S$. Conversely every $2$-regulus of $\S$ corresponds to
an \orsl\ of $\li$. 
\item If $b$ meets $\li$ in a point, then in $\PG(6,q)$, $b$  corresponds to
  a line of $\PG(6,q)\backslash\si$. Conversely every line of
  $\PG(6,q)\backslash\si$ corresponds to an \orsl\ of $\PG(2,q^3)$ tangent
  to $\li$.
\item\Label{FFA-orsl}
If $b$ is disjoint from $\li$, then in
$\PG(6,q)$, $b$ corresponds to a special twisted cubic. Further, a
twisted cubic $\N$ of $\PG(6,q)$ corresponds to an \orsl\ of
$\PG(2,q^3)$ if and only if $\N$ is special.
\end{enumerate}
\end{theorem}

\begin{theorem}{\rm \cite{barw12}}\Label{FFA-subplane-tangent}
Let $\Bpi$ be an \orsp\ of $\PG(2,\r^3)$.
\begin{enumerate}
\item If $\Bpi$ is secant to $\li$, then
  in $\PG(6,q)$,
  $\Bpi$ corresponds to a plane of $\PG(6,q)$ that meets $q+1$ spread
  elements. Conversely, any plane of $\PG(6,q)$ that meets $q+1$ spread
  elements corresponds to an \orsp\ of $\PG(2,q^3)$ secant to $\li$. 
\item  Suppose  $\Bpi$  is tangent to $\li$ in the
  point $T$. 
Then $\Bpi$ determines a set $\bbb$  of points in $\PG(6,\r)$ (where the
affine points of $\Bpi$ correspond to the affine points of $\bbb$) such that:
\begin{enumerate}
\item[{\rm (a)}]  $\bbb$ is a ruled surface with conic directrix $\bbc$ contained in
  the plane $[T]\in\S$, and twisted cubic directrix $\bbn$ contained in a
  3-space $\Sigma$ that meets $\si$ in a spread element (distinct from
  $[T]$). The points of $\bbb$ lie on $\r+1$ pairwise disjoint generator lines joining $\bbc$ to $\bbn$.
\item[{\rm (b)}]  The $\r+1$ generator lines of $\bbb$ joining $\bbc$
  to $\bbn$ are determined by a projectivity from $\bbc$ to $\bbn$.
\item[{\rm (c)}]  When we extend $\bbb$
  to $\PG(6,\r^3)$, it contains the conjugate transversal
  lines $g,g^\r,g^{\r^2}$ of the regular 2-spread $\S$. So $\C$ and $\N$ are special.
\item[{\rm (d)}]  $\bbb$ is the intersection of nine quadrics in $\PG(6,\r)$.
\end{enumerate}
\end{enumerate}
\end{theorem}

\subsection[Collineations]{The collineation group in the Bruck-Bose
  representation}

Consider a homography $\alpha\in\PGL(3,\r^3)$, acting on $\PG(2,\r^3)$, which fixes
$\li$ as a set of points. There is a corresponding homography
$[\alpha]\in\PGL(7,\r)$ acting on the Bruck Bose representation of
$\PG(2,\r^3)$ as $\PG(6,\r)$. Note that $[\alpha]$ fixes the hyperplane
$\si$ at infinity of $\PG(6,\r)$, and permutes the elements of the regular
spread $\S$ in
$\si$.  Consequently subgroups of $\PGL(3,\r^3)$ fixing $\li$ correspond to
subgroups of $\PGL(7,\r)$ fixing $\si$ and permuting the spread elements in
$\si$. 
For more details, see \cite{barwtgt2}.
Hence when we prove results about transitivity in $\PG(2,q^3)$, there is a
corresponding transitivity result in $\PG(6,q)$.
We are interested in a particular Singer cycle acting on $\PG(6,q)$. It is
straightforward to prove the following.

\begin{theorem}\Label{PGL7r}
Consider  the homography $\CapitalTheta\in\PGL(7,\r)$ with $7\times 7$ matrix
$$M=\begin{pmatrix}T&0&0\\0&T&0\\0&0&1\end{pmatrix}, \quad{\rm where}\ \ 
T=\begin{pmatrix}0&0&t_0\\1&0&t_1\\0&1&t_2\end{pmatrix}.$$
So a point
Then in $\PG(6,q)$, $\CapitalTheta$ fixes each plane of the 2-spread $\S$, and $\langle\CapitalTheta\rangle$
acts regularly on the set of points, and on the set of lines of each spread element.
\end{theorem}

\section{Sublines contained in a tangent splash}\Label{sect:sublines}

Let $\pi$ be an \orsp\ of $\PG(2,q^3)$ tangent to $\li$ at the point $T$. 
Recall that the \emph{tangent splash} $\ST$ of $\pi$ is the set of points
of $\li$ lying on lines of $\pi$, and $T=\pi\cap\li$ is called the {\em centre} of $\ST$. 
As two distinct lines of $\pi$ meet in a
point of $\pi$, the lines of $\pi$ not through $T$ meet $\li$ in distinct
points. Hence
a tangent splash  has $q^2+1$ points. 
We say a set of $q^2+1$ points $\ST\subset\li$ is a tangent splash with centre $T$ 
if there is an \orsp\ $\pi$ with
 tangent splash $\ST$ and $\pi\cap\li=T$.
 We will show in Theorem~\ref{uniquecentre} that all \orsps\ with a given tangent splash have the same centre, that is, a tangent splash has a unique centre. 
 Note that it also makes sense to talk about the {\em tangent splash of $\pi$ onto the line $\ell$} where $\pi$ is an \orsp\ tangent to the line $\ell$.

In this section we investigate the \orsls\ contained in a tangent splash $\ST$
and show that the points of a tangent splash form a Desarguesian affine plane. 
The next result investigates \orsls\ contained in $\ST$ which contain
the centre $T$. We show later in Corollary~\ref{nos6} that every \orsl\ contained in $\ST$
must contain the centre $T$.

\begin{lemma}\Label{numberofsublines}
Let $\pi$ be an \orsp\ of $\PG(2,q^3)$ tangent to $\li$ at the point $T$ with tangent splash $\ST$.
\begin{enumerate}
\item \Label{sublineprojection} 
The lines of $\pi$ through an affine point $P\in\pi$ meet $\ST$ in an
\orsl\ $t_P$.  
Further, there is a bijection between the $q^2+q$ affine
  points $P$ of $\pi$  and the \orsls\  $t_P$ containing $T$ and contained in $\ST$.
\item \Label{tgtsplash2pts} Let $U,V$ be distinct points non-centre points of $\ST$. Then the unique \orsl\ containing $T,U,V$ is contained in $\ST$. 
\end{enumerate}
\end{lemma}

\begin{proof} 
Label the lines of $\pi$ through $P$ by
$\ell_0=PT,\ell_1,\ldots,\ell_q$. For each $i$, let $\overline\ell_i$ denote the
extension of $\ell_i$ to $\PG(2,q^3)$. 
If $m$ is another line of $\pi$, $P\notin m$, then the points $\ell_0\cap
m,\ldots,\ell_q\cap m$ are the $q+1$ distinct points of $m$. Hence the points
$t_P=\{T=\overline\ell_0\cap\li,\overline\ell_1\cap\li,\ldots,\overline\ell_q\cap\li\}$ are the projection from $P$ of an \orsl\
$m$ onto the line $\li$, and so $t_P$ is an \orsl\ of $\li$. As the \orsls\
$\ell_i$ are lines of $\pi$, the points $\overline\ell_i\cap\li$ are points of
$\ST$, so $t_P$ is contained in $\ST$ and contains the centre $T$. 

As two lines of $\pi$ meet in $\pi$, through a non-centre point $U$ of $\ST$ there
is unique line $u$ of
$\pi$. So an \orsl\ of $\ST$ containing $T$ is projected by at most one point of
$\pi$. 
Hence for distinct affine points $P,Q$ of $\pi$, the \orsls\ $t_P$ and $t_Q$ are distinct.
There are $q^2+q$ points of $\pi$ distinct from $T$. We show in the next
paragraph that there are $q^2+q$ \orsls\  contained in $\ST$ and containing $T$, and hence
we have a bijection from affine points $P\in\pi$ to \orsls\ $t_P$ contained
in $\ST$ and containing $T$.

Let $U,V$ be distinct points of $\ST$, $U,V\neq T$, and let $u,v$ be the distinct lines in $\pi$ meeting $\li$ in
$U,V$ respectively.  Let $P=u\cap v$, so $P\in\pi$ and $P\ne T$ as
$T\not\in u,v$. 
The lines of $\pi$ through $P$ meet $\li$ in an \orsl\ $t_P$ containing 
$T,U,V$. Hence $t_P$ is the unique \orsl\ through $T,U,V$. By part 1, $t_P$
is contained in $\ST$, completing the proof of  part \ref{tgtsplash2pts}.
As there are $\binom{\r^2}2/\binom \r 2$ ways to choose $U,V$ to get a
unique \orsl, $\ST$ contains exactly $q(q+1)$ \orsls\ containing $T$,
 completing the proof of part 1.
\end{proof}

Lemma~\ref{numberofsublines} has a natural correspondence in
the Bruck-Bose representation in $\PG(6,q)$. For example, part 2 says that
given a tangent splash $[\ST]$ which consists of $q^2+1$
spread elements, the unique 2-regulus containing the centre $[T]$ and two
further elements $[U],[V]$ of $[\ST]$ is contained in $[\ST]$.

The next result shows how to construct an affine plane from the points and
\orsls\ of a tangent splash. 

\begin{theorem}\Label{splash-is-affine-plane}
Consider the incidence structure $\mathcal I$ with {\sl points} the 
$q^2$ elements of $\ST\backslash\{T\}$, {\sl lines} 
the \orsls\ containing $T$ and contained in $\ST$, and {\sl incidence} is inclusion.
Then $\mathcal I$ is a Desarguesian affine plane of order $\r$. 
\end{theorem}
\begin{proof} 
By Lemma~\ref{numberofsublines}(\ref{tgtsplash2pts}), two points of $\mathcal
I$ lie in a unique line of $\mathcal I$. Hence $\mathcal I$ is a
2-$(q^2,q,1)$ design and so is an affine plane of order $q$. 
Further, by \cite[Theorem 4.10]{hughes}, the design consisting of the \orsls\ of $\li$ through $T$ is the affine geometry AG$(3,q)$. Moreover, $\mathcal I$ is an affine subplane of AG$(3,q)$, and so $\mathcal I$ is Desarguesian.
\end{proof} 

\section[Group properties]{Group properties of tangent subplanes and
  tangent splashes}\Label{sect:group}

In this section we study the collineation group of $\PG(2,q^3)$ fixing
$\li$ and show it is transitive on \orsps\ tangent to $\li$, and hence transitive
on tangent splashes on $\li$. 
Thus all tangent subplanes have the same geometrical and group properties,
and to prove results about tangent subplanes in general, we may
coordinatise a particular tangent subplane, and work with that. We begin by showing that a tangent splash has a unique centre.

\begin{theorem}\Label{uniquecentre} Let $\pi$ be an \orsp\ with tangent splash $\ST$ on  tangent line $\ell$ and centre $T$. 
Then any \orsp\ tangent to $\ell$ that has tangent splash $\ST$ on $\ell$ has centre $T$, that is, 
 a tangent splash has a unique centre.
\end{theorem}

\begin{proof}
Let $G$ be the full stabilizer of $\ST$ in the group $\PGL(3,q^3)_\ell$. Then $G$ contains a subgroup $I$ of central collineations with centre $T$ that fixes $\pi$, and we note that $|I|=q^2(q-1)$.
The subgroup of order $q^2$ in $I$ is semiregular on the $q^2$ points of $\ST\setminus\{T\}$, so these points form an orbit in $I$. 

Suppose that $\pi'$ is an \orsp\ tangent to $\ell$ with tangent splash $\ST$, and centre $T'\neq T$. Then the corresponding group of central collineations $I'$ cannot fix $T$. Hence the $q^2+1$ points of $\ST$ must be an orbit of $G$, hence $q^2+1$ must divide $|\PGL(3,q^3)_\ell|=q^9(q^3-1)^2(q^3+1)$. That is, $q^2+1$ divides $(q^3-1)^2(q^3+1)$, and so any odd prime divisor of $q^2+1$ must divide $(q^2+q+1)(q^3+1)$, a contradiction. 
Hence $T'=T$ as required. 
\end{proof}

The next result is the main result of this section, it shows that all tangent splashes are projectively equivalent. 

\begin{theorem} \Label{transitiveontangentsubplanes}
The subgroup of $\PGL(3,q^3)$ acting on $\PG(2,q^3)$ and fixing a line $\ell$  is transitive on all the \orsps\  meeting $\ell$ in a point, and hence is transitive on the tangent splashes of $\ell$. \end{theorem}

\begin{proof}
Firstly, note that any collineation of $\PGL(3,q^3)$ fixing pointwise three points on a line, fixes the entire line, and hence is a central collineation.  The result also holds for a collineation fixing linewise three lines through a point.
 Without loss of generality we prove the result for the \orsp\ 
$\pi=\PG(2,q)$ of $\PG(2,q^3)$. 
Fix any point $T$ of $\pi$ and let $\ell(T)$ be the set of lines of $\PG(2,q^3)$ that are tangent to $\pi$ at $T$. As $\PGL(3,q^3)$ acts faithfully on $\PG(2,q)$, we have $\PGL(3,q^3)_\pi=\PGL(3,q)$. Let $G$ be the subgroup $\PGL(3,q)_T$. 
 As $\PGL(3,q)$ is point transitive on the $q^2+q+1$ points of $\pi$, the orbit-stabilizer theorem gives $|G|=|\PGL(3,q)|/(q^2+q+1)=q^3(q^2-1)(q-1)$. 

Let $I$ be the subgroup of $G$ of central collineations with centre $T$, so $|I|=q^2(q-1)$. Fix a line $\ell$ of $\ell(T)$, and consider the group $H=G{_\ell}=\PGL(3,q)_{T,\ell}$.  We will show that $H=I$.  As $I$ fixes all the lines of $\pi$ through $T$, we have $H\geq I$. Note that $|G/I|=q(q+1)(q-1)$, so any collineation in $H$ has order dividing this number. We will prove: (1) the only central collineations in $H$ are in $I$, and so any collineation in $H$ fixing three points on $\ell$ or three lines through $T$ is in $I$; (2) No collineation in $H\backslash I$ has order $p\mid q$; (3) No collineation in $H\backslash I$ has odd order dividing $q+1$; and (4) Any Sylow 2-group of $H$ lies in $I$.  From these four assertions we conclude that $H=I$.

 To prove (1), if $\alpha$ is a central collineation in $H$, then for each non-fixed point $P\in\pi$, the line $PP^\alpha$ is a fixed line.  So the centre, and hence the axis, of $\alpha$ lies in $\pi$ and so $\alpha\in I$.  For (2), if $\alpha\in H$ has prime order $p\mid q$ then as $\alpha$ fixes $\ell$ and gcd$(p,q+1)=1$ and gcd$(p,q^3-q-1)=1$ it follows that $\alpha$ fixes at least one line of $\pi$ through $P$ and another line of $\ell(T)$ (other than $\ell$), and so $\alpha$ fixes three lines through $T$, and by the note above, $\alpha\in I$.  For (3), if $\alpha\in H\backslash I$ has odd prime order $x\mid q+1$ then gcd$(x,q^3)=1$, $\alpha$ fixes another point $Q\ne P$ on $\ell$, and as $x$ is odd, gcd$(x,q^^3-1)=1$ so $\alpha$ fixes at least three points of $\ell$, and so $\alpha\in I$. For (4), first consider a Sylow $x$-subgroup $S$ of $H$ not lying in $I$, for $x=2$.  From (2), we may assume that $q$ is odd. $S$ fixes another point on $\ell$ and since $S\not\subseteq I$, it acts semi-regularly on the remaining $q^3-1$ points of $\ell$.  Thus $|S|\mid q^3-1$.  As $q^3-1$ is an even number $q-1$ times an odd number $q^2+q+1$, it follows that $|S|\mid q-1$, and so $S\subseteq I$, a contradiction.  Thus $x$ is odd, and by (3)  $x\nmid q+1$, and using (2) we have $x\mid q-1$.  Now consider the action of a Sylow $x$-subgroup of $H/I$ on the lines through $T$.  As there are $q+1$ lines of $\pi$ through $T$ and $(q+1)\bmod (q-1)=2$, it follows that $U$ fixes $\ell$ and at least two lines through $T$, and hence all the lines through $T$, a contradiction.  Thus we have proved (1) -- (4) and so $H=I$.

Thus we have shown that $I=\PGL(3,q^3)_{\pi,\ell}=\PGL(3,q)_\ell$.
As $|\ell(T)|=q^3-q=|G/I|$, the semiregular group $G/I$ is regular. Thus $G$ is transitive on the $q^3-q$ lines in $\ell(T)$, and $G$ faithfully induces the regular group $G/I$ on $\ell(T)$. 

Let $\pi=\PG(2,q)$ and $\pi'$ be two \orsps\ meeting $\ell$ in $T,T'$ respectively. Let $m,n$ be lines in $\pi$ through $T$, and choose a quadrangle $A,B\in m\setminus\{T\}$, $C,D\in n\setminus\{T\}$ in $\pi$. 
In a similar way, choose a quadrangle $A'B'C'D'$ in $\pi'$, so that $A'B'\cap\ell=C'D'\cap\ell=T'$.
So we can apply a collineation $\phi:A'B'C'D' \mapsto ABCD$, and $\phi$ maps $\pi'$ to $\pi$ and $T'$ to $T$. Now $\phi$ maps the line $\ell$ to a line $m\in\ell(T)$. So the result follows from the fact that $G/I$ is regular on $\ell(T)$. 
\end{proof}

The following corollary follows directly from this proof, and is useful in the rest of the article.

\begin{corollary}\Label{G9} Let $\pi$ be an \orsp\ tangent to $\ell$ in $\PG(2,q^3)$. The subgroup 
$I=\PGL(3,q^3)_{\pi,\ell}$ fixes the tangent splash of $\pi$ and is transitive
  on the non-centre points of this splash. Further, $I$
fixes linewise the lines of $\pi$ through
  $T=\pi\cap\ell$,  is transitive on the points of $\pi\bs\{T\}$ on
  these lines, and is transitive on the lines of $\pi$ not through $T$.
  \end{corollary}
  
  It will also be useful to study the group of collineations that fixes an \orsp, and determine the orbits of points and lines of $\PG(2,q^3)$ in this group. 
  
  \begin{lemma}\Label{Gextsubgp}
  Let $\pi$ be an \orsp\ of $\PG(2,q^3)$, and let $K=\PGL(3,q^3)_\pi$. Then $K$  has three orbits on the points of
  $\PG(2,\r^3)$: the points of $\pi$, the points of
  $\PG(2,q^3)\bs\pi$ that lie on a line of $\pi$,  and the points
   that do not lie on a line of $\pi$. $K$ has
  three line orbits: lines of $\pi$, lines of $\PG(2,\r^3)$ tangent to
  $\pi$, and lines of $\PG(2,\r^3)$ exterior to $\pi$.
  \end{lemma}
  
  \begin{proof}
   Without loss of generality, we prove this for the \orsp\ $\pi=\PG(2,q)$, so $K=\PGL(3,q^3)_\pi=\PGL(3,q)$.
  We will show that $K$ is transitive on the set of lines exterior to $\pi$. To do this we use the exterior splash of $\pi$ onto an exterior line. 
Let $\ell,m$ be two lines exterior to $\pi$. 
Let $\Sl$ be the exterior splash of $\pi$ onto $\ell$, and let $\Sm$ be the exterior splash of $\pi$ onto $m$. Recall that an exterior splash has size $q^2+q+1$.  

Suppose first that  $P=\ell\cap m$ lies in both $\Sl$ and $\Sm$. There is a unique line $\ell_P$ of $\pi$ through $P$. The group of central collineations $H<K$ with axis $\ell_P$ has order $q^2(q-1)$ and $H$ is semiregular on the exterior lines through $P$. However, the number of lines exterior to $\pi$ through $P$ is also $q^2(q-1)$. Hence if $P\in\Sl \cap\Sm$, then all the exterior lines through $P$ are in the same $K$-orbit.

Now suppose  $P=\ell\cap m$ is not in $\Sl \cap\Sm$. Through any point $L\in\Sl$, there are $q^2$ tangents and one secant line of $\pi$. So there is an exterior line $t$ through $L$ that contains a point $M$ of $\Sm$ (as there are $(q^2+q+1)-(q^2+1) >0$ choices for $M$).   By the above argument, $\ell,t$ lie in the same orbit of $K$, and $m,t$ lie in the same orbit of $K$. 
Hence $K$ is transitive on the set of lines exterior to $\pi$. Note that $K$ is transitive on the points of $\pi$ and using results from the proof of Theorem~\ref{transitiveontangentsubplanes}, we conclude that $K$ is transitive on lines tangent to $\pi$.
\end{proof}

  We conclude this section with the next lemma which contains some elementary counting results in $\PG(2,q^3)$. It is useful to have these stated for this article. The proof is straightforward. 
  
  \begin{lemma} \Label{countinglemma}  Let $\pi$ be an \orsp\ in $\PG(2,q^3)$:
\begin{enumerate}
\item \Label{G0} There are $\r^6(\r^6+\r^3+1)(\r^2-\r+1)(\r^2+\r+1)$ \orsps.
\item \Label{G2} $\pi$ has $\r^2+\r+1$ secant lines, $(\r^2+\r+1)\r(\r^2-1)$ tangent lines, and  $\r^3(\r^2-1)(\r-1)$ exterior lines. 
\item \Label{G3} Through each point $P\notin\pi$ there are either one secant line and $\r^2$ tangent lines, or $\r^2+\r+1$ tangent lines.  
\item \Label{G4} The number of \orsps\  tangent to a given line is $\r^7(\r^3-1)(\r^3+1)(\r^2+\r+1)$.
\end{enumerate}
\end{lemma}

\section{Counting tangent splashes}\Label{sect:count}

In this section we count the number of tangent splashes on $\li$. Further we
investigate the number of points 
needed to determine a unique \orsp\ with a given tangent splash.

\begin{theorem}\Label{splashextsubplane}
Let $\ST$ be a tangent splash of $\li$ and let $\ell$ be an \orsl\ disjoint from
$\li$ lying on a line of $\PG(2,q^3)$ which meets
$\ST\backslash\{T\}$. Then there is a unique tangent \orsp\ that contains $\ell$
and has tangent splash $\ST$.
\end{theorem}

\begin{proof}
We prove this result using coordinates. Everywhere else  in this article we let the
line at infinity $\li$ of $\PG(2,q^3)$ have homogeneous coordinates $[0,0,1]$. Here we give the line at
infinity 
the coordinates $[0,1,\tau]$, and to avoid confusion, denote it
by
$\li'$. Consider the \orsp\ $\pi=\PG(2,q)$, it is tangent to $\li'$
at the point $T=(1,0,0)$. The lines of $\pi$ not through $T$ have
coordinates $[1,r,s]$, $r,s\in\GF(q)$, and meet $\li'$ in the points
$(s-r\tau,\tau,-1)$. So the tangent splash of $\pi$ onto $\li'$ is
$\ST=\{T=(1,0,0)\}\cup\{(s-r\tau,\tau,-1)\st r,s\in\GF(q)\}$. Note that the
set $\ell=\{(0,c,1)\st c\in\GF(q)\cup\{\infty\}\}$ is an \orsl\ of $\pi$ that
is disjoint from $\li'$. 

By Theorem~\ref{transitiveontangentsubplanes} and Corollary~\ref{G9}, without loss
of generality we only need prove the theorem holds for the splash
$\ST=\{T=(1,0,0)\}\cup\{(s-r\tau,\tau,-1)\st r,s\in\GF(q)\}$ and the \orsl\
$\ell=\{(0,c,1)\st c\in\GF(q)\cup\{\infty\}\}$. It suffices to show
that the only \orsp\ containing $\ell$ with tangent splash $\ST$ is
$\pi=\PG(2,q)$. 

Any point in an \orsp\ containing $T$ and $\ell$ will be on a line joining
$T$ to a point of $\ell$. Such a point has coordinates $(\alpha,1,0)$ or
$(\alpha,c,1)$ for $\alpha\in\GF(q^3)$, $c\in\GF(q)$. 
Consider one such point $X=(\alpha,d,1)$ for some $\alpha\in\GF(q^3)$,
$d\in\GF(q)$. The set of points $T,X,\ell$ contains a quadrangle and so lie
in a unique \orsp\ $\pi'$. We show that $\pi'$ has splash
$\ST=\{T=(1,0,0)\}\cup\{(s-r\tau,\tau,-1)\st r,s\in\GF(q)\}$ if and only if
$\alpha\in\GF(q)$, and so $\pi'=\PG(2,q)$.

Let $m$ be
 the unique \orsl\ determined by
$T=(1,0,0)$, $X=(\alpha,d,1)$, and $(0,d,1)$.
The homography with matrix
$$
\begin{pmatrix}
\alpha&0&0\\0&d&0\\0&1&1
\end{pmatrix}
$$
maps the \orsl\ $\{(a,1,0)\st a\in\GF(q)\cup\{\infty\}\}$ onto the \orsl\ $m$, so 
$m$   has points $X_a=(a\alpha,d,1)$ where $a\in\GF(\r)\cup\{\infty\}$. 
The line joining $X_a$ to the point $(0,c,1)$ of $\ell$
meets the line $\li'$ in the point
$(-a\alpha\tau-a\alpha c,-\tau(d-c),d-c)$. As $a,c,d\in\GF(q)$, this point belongs to the
tangent splash $\ST$ only if $\alpha\tau+\alpha c=s-r\tau$ for some $r,s\in\GF(q)$. 
Writing $\alpha=\alpha_0+\alpha_1\tau+\alpha_2\tau^2$ for unique $\alpha_0,\alpha_1,\alpha_2\in\GF(\r)$, we get
\begin{eqnarray*}
\alpha\tau+\alpha c=
(\alpha_2t_0+c\alpha_0)+(\alpha_0 +\alpha_2 t_1+c\alpha_1)\tau+(\alpha_1 +\alpha_2 t_2+c\alpha_2)\tau^2.
\end{eqnarray*}
We want the coefficient of $\tau^2$ to be zero for all
$c\in\GF(q)$, that is, $\alpha_1 +\alpha_2 t_2+c\alpha_2=0$. Setting 
$c=-t_2$ gives $\alpha_1=0$, and setting $c=-t_2+1$ gives $\alpha_2=0$.
 Hence $\alpha=\alpha_0$ is in $\GF(q)$, and so the point
$X_a$ lies in $\PG(2,q)$ for all $a\in\GF(q)$. 
Hence the \orsp\ $\pi'$ has splash $\ST$ if and only if $\pi'=\PG(2,q)$. Hence
 there is a unique \orsp\ (namely
 $\PG(2,q)$)
that contains $T$, $\ell$, and has tangent splash $\ST$.
\end{proof}

In \cite{barwtgt2}, it is shown how to  geometrically construct the unique \orsp\
from a tangent splash of $\li$ and an \orsl\ disjoint from  $\li$.

We now count the number of tangent splashes in $\PG(2,q^3)$.

\begin{theorem}\Label{notgtsplash}
\begin{enumerate}
\item \Label{C-3} The number of  tangent \orsps\  with a given tangent splash is $\r^6(\r^3-1)$.
\item \Label{C-2} The number of tangent splashes on $\li$ is $\r(\r^3+1)(\r^2+\r+1)$.
\item \Label{C1} The number of tangent splashes of $\li$ with a common centre is $\r(\r^2+\r+1)$.
 \end{enumerate}
\end{theorem}

\begin{proof}
To prove part \ref{C-3}, 
let $\ST$ be a tangent splash on the line $\li$ in $\PG(2,q^3)$. 
We count the pairs $(\ell,\pi)$ where $\pi$ is an \orsp\ tangent
to $\li$ at $T$ with splash $\ST$ and $\ell$ is line of $\pi$ not
through $T$.  
First note that if $m$ is a line of $\PG(2,q^3)$ and $X$ a point of $m$,
then it is straightforward to show that the number of \orsls\ of $m$ that do
not contain $X$ is $q^3(q^3-1)$. 
There are $q^2$ points in $\ST$ distinct from $T$, and each point lies on
$q^3$ lines of $\PG(2,q^3)$ distinct from $\li$.
Hence there are $q^2\times q^3\times q^3(q^3-1)$ choices for an \orsl\ $\ell$. 
By Theorem~\ref{splashextsubplane}, each of these \orsls\ $\ell$ lies on a unique
\orsp\ that contains $\ell$ and has tangent splash $\ST$. 
  Hence if $n$ is the number of \orsps\ with the given tangent splash, we
  have $\r^2\times\r^3\times\r^3(\r^3-1)\times 1=n\times \r^2$ and so there
  are $\r^6(\r^3-1)$ tangent \orsps\ with a given tangent splash.

To count the number of tangent splashes on $\li$,  by
Theorem~\ref{transitiveontangentsubplanes} we can divide the total
number of \orsps\ tangent to $\li$
(calculated in Lemma~\ref{countinglemma}) by the
number of \orsps\ with a given tangent splash (calculated in part
\ref{C-3}). This proves part \ref{C-2}.

For part \ref{C1}, the subgroup of $\PGL(3,\r^3)$ fixing $\li$ is transitive on the points of $\li$, hence each point of $\li$ is the centre of a constant number of splashes.  The result now follows from part  \ref{C-2}.
\end{proof}

Theorem~\ref{splashextsubplane} shows that a tangent splash and one affine \orsl\
not through the centre uniquely determines an \orsp. We now consider the case
of a tangent splash and one affine \orsl\ through the centre, and show that this
determines more than one \orsp.

\begin{theorem}
In $\PG(2,q^3)$, let $\ST$ be a tangent splash of $\li$ with centre $T$. 
Let $m$ be an \orsl\ through $T$ (not in $\li$). Then there are
  $q(q^2-1)$ tangent \orsp s that contain $m$ and have tangent splash $\ST$.
\end{theorem}

\begin{proof}
First note that the subgroup of $\PGL(3,q^3)$ containing translations with
axis $\li$ is 2-transitive on the affine points of an affine line. Hence
the number of \orsps\ containing two distinct affine points $P,Q$ such that
$T\in PQ$ is a
constant, denote it by $x$. As $T,P,Q$ lie in a unique \orsl, $x$ is the number
of \orsps\ containing an affine \orsl\ through $T$. We count in two ways the
triples $(P,Q,\pi)$ where $P,Q$ are distinct affine points, the line $PQ$
contains $T$ and $\pi$ is a tangent \orsp\ with splash $\ST$ containing
$P$ and $Q$. Using Lemma~\ref{notgtsplash}(\ref{C-3}) we have
$\r^6(\r^3-1)\times x=\r^6(\r^3-1)\times
(\r^2+\r)(\r-1)$ and so $x=\r(\r^2-1)$.
\end{proof}

We now look at how many points of $\li$ are needed to uniquely determine a
tangent splash.
 

\begin{theorem}\Label{nos5-version2} Two tangent splashes with a common centre $T$ can
  meet in at most an \orsl\ which contains $T$.  Further, there are $q+1$
  tangent splashes of $\li$ with centre $T$ containing a fixed \orsl\ through $T$.
\end{theorem}

\begin{proof}
 Suppose two splashes $\ST$ and $\ST'$ have a common centre $T$ and contain
 three further common points $U,V,W$ not all on \orsl\ through $T$.  
Denote by $\ell(X,Y,Z)$ the unique \orsl\ through any three distinct
points $X,Y,Z$ of $\li$.  
By Lemma~\ref{numberofsublines}(\ref{tgtsplash2pts}),
 $\ST$ and $\ST'$ both contain the \orsl\ $\ell(T,U,V)=\{T,U_1=U,U_2,\ldots,U_\r\}$. Further,
 the $\r$ \orsls\ $\ell(T,U_i,W)$, $i=1,\ldots,q$, are 
contained in both $\ST$ and $\SdT$.  These $q$ \orsls\ together with
$\ell(T,U,V)$ cover $(\r+1)+\r\times (\r-1)=\r^2+1$ points.  Hence $\ST=\ST'$.
Thus $\ST$, $\ST'$ can both contain the \orsl\ $\ell(T,U,V)$, but cannot
have any further points in common.

Let $x$ be the number of splashes with centre $T$ containing
a fixed \orsl\ $\ell$ of $\li$, where $\ell$ contains $T$. We count in two ways the pairs $(\ell,\ST)$ where
$\ell$ is an \orsl\ of $\li$ through $T$ and $\ST$ is a tangent splash with
centre $T$ containing $\ell$.
By Lemma~\ref{numberofsublines}(\ref{sublineprojection}), the number of
\orsls\ contained in $\ST$ and containing $T$ is $q^2+q$. So using 
 Theorem~\ref{notgtsplash}(\ref{C1}),
 we have 
$$ q^2(q^2+q+1)\times x= q(q^2+q+1)\times q(q+1),$$
so $x=q+1$.
That is, there are $q+1$ tangent
splashes with common centre $T$ and containing a common \orsl\ $\ell$ through
$T$.
\end{proof}

\begin{theorem}\Label{nos5-version2-part2}
 Let $T,U,V,W$ be four distinct points of $\li$ not on a
  common \orsl. Then
there is a unique tangent splash containing $T,U,V,W$ with centre $T$.
\end{theorem}

\begin{proof}
Let $\ell$ be the unique \orsl\ containing $T,U,V$. 
By Lemma~\ref{numberofsublines}(\ref{tgtsplash2pts}), a tangent splash with
centre $T$ containing $U,V$ must contain $\ell$. 
By
Theorem~\ref{nos5-version2}, there are $q+1$ tangent splashes with centre
$T$ containing $\ell$. These tangent splashes cover $(q^2-q)\times(q+1)+q+1=q^3+1$
points of $\li$. Hence $W$ lies in exactly one of these tangent splashes.
Hence $T$ and three further points of $\li$ not all on a common \orsl\ lie
in a unique tangent splash with centre $T$.
\end{proof}

\section[Coordinatisation]{Coordinatisation of a tangent subplane}\Label{coord-subplane}


We need to calculate the coordinates of points in an \orsp\ of
$\PG(2,q^3)$ that is tangent to $\li$. 
The \orsp\ $\PG(2,q)$ of $\PG(2,q^3)$ meets $\li$ in $q+1$ points. By
\cite[Lemma 2.6]{barw12}, we can
map $\PG(2,q)$ to an \orsp\ $\Bpi$ tangent to $\li$ using the homography $\zeta$ with matrix
$A_1$, where
\[
A_1=\begin{pmatrix}
-\w&1+\w&0\\
0&1&0\\
0&1+\w&-\w
\end{pmatrix},
\qquad
A_1'=
\begin{pmatrix}
-1&1+\w&0\\
0&\w&0\\
0&1+\w&-1
\end{pmatrix}.
\]
The matrix $A_1'$ is the matrix of the inverse homography $\zeta^{-1}$
that maps $\Bpi$ to $\PG(2,q)$. 

\begin{figure}[h]
\centering
\input{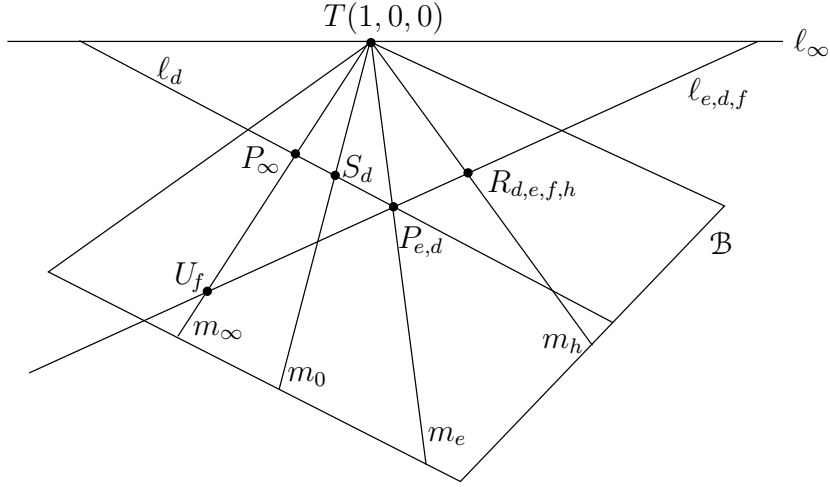}
\caption{Tangent subplane notation}\label{fig-subscripts}
\end{figure}

Note that $\zeta$ fixes the points $T=(1,0,0)$ and 
$P_\infty=(1,1,1)$ of $\PG(2,\r)$.
 The lines of $\PG(2,q)$  through $T$ can be written with homogeneous coordinates $m'_e=[0,1-e,e]$,
 $e\in\GF(\r)$ and $m'_\infty=TP_\infty=[0,1,-1]$. For $e\in\GF(q)$, the line $m'_e$ is mapped by $\zeta$ to
 the line $m_e=[0,e+\tau,-e]$ of $\Bpi$ (a line through
 $T$).  The line $m'_\infty$ of $\PG(2,\r)$ is fixed by $\zeta$ and so $m_\infty=[0,1,-1]$ is the final line of $\Bpi$
 through $T$.

\begin{table}[ht]
\caption{Coordinates in the tangent \orsp\ $\Bpi$, ($e,d,f,h\in\GF(q)$)}\label{table-coords}
\begin{center}
\begin{tabular}{|c|c|c|}
\hline
Notation&Coordinates&Description\\
\hline
$T$&$(1,0,0)$&$\Bpi\cap\li$\\
$P_\infty$&$(1,1,1)$&\\
$m_e$&$[0,e+\tau,-e]$&lines of $\Bpi$ through $T$\\
$m_\infty$&$[0,1,-1]$&$TP_\infty$\\
$S_d$&$(d,0,1)$&points of $\Bpi$ on $m_0$\\
$U_f$&$(1+f\tau,1,1)$&points of $\Bpi$ on $m_\infty$\\
$\ell_d$&$[1,d-1,-d]$&$P_\infty S_d$\\
$P_{e,d}$&$(e+d\tau,e,e+\tau)$&$m_e\cap \ell_d$\\
$\ell_{e,d,f}$&$[-1,ef-d+1+f\tau,d-ef]$&$P_{e,d}U_f$\\
$R_{e,d,f,h}$&$(h+(fh-fe+d)\tau,h,h+\tau)$& $\ell_{e,d,f}\cap m_h$\\
$R_{e,d,f,\infty}$&$(1+f\tau,1,1)$&$\ell_{e,d,f}\cap m_\infty=U_f$\\
\hline
\end{tabular}
\end{center}
\end{table}

Next we find the coordinates of the  points in $\Bpi$,
see
Figure~\ref{fig-subscripts} for an illustration of the notation used. First
consider the line $m_\infty'$ of $\PG(2,q)$. 
For $f\in\GF(q)$,  the point $U'_f=(1-f,1,1)\in m_\infty'$ is mapped
under $\zeta$ to the point $U_f=(1+f\tau,1,1)$ on $m_\infty$ in $\Bpi$. 
Now consider the line $m'_0$ of $\PG(2,q)$, it has points $T$ and
$S'_d=(d,0,1)$, $d\in\GF(q)$.  The point $S_d'$ maps under $\zeta$ to
$S_d=(d,0,1)$.  So the affine points of $m_0$ in $\Bpi$ are $S_d$, $d\in\GF(\r)$.

We work out the coordinates of the other points of $\Bpi$ as follows. 
 First define the $q$ lines $\ell_d=P_\infty S_d$, $d\in\GF(\r)$.
Now for $d,e,f,h\in\GF(q)$, let $P_{e,d}=\ell_d\cap m_e$, $\ell_{e,d,f}=P_{e,d}U_f$,
and $R_{d,e,f,h}=\ell_{e,d,f}\cap m_h$, $R_{e,d,f,\infty}=\ell_{e,d,f}\cap m_\infty=U_f$.
 Working out the coordinates for
these points and lines is straightforward, and they are given in Table~\ref{table-coords}.

The  notation $\ell_{e,d,f}$ is needed in subsequent calculations. However, if we wish to enumerate the lines of $\Bpi$, we can use the following lemma.

\begin{lemma}\Label{allextlines}
Every line in $\Bpi$ not through $T$ is represented by some $\ell_{0,d,f}$ with $d,f\in\GF(\r)$.
\end{lemma}

\begin{proof}
Consider the image of the general affine line $t'=[1,a,b]$ of $\PG(2,\r)$
not through $T=(1,0,0)$.  Under $\zeta$ it is mapped to
$t=[-1,1+b+(1+a+b)\tau,-b]$.  Calculating $\ell_{e,d,f}$ with
$e=0,d=-b,f=1+a+b$ gives $[-1,0\times f+b+1+(1+a+b)\tau,-b]$    which is the
line $t$ as required.
\end{proof}

We can use this to calculate the coordinates of the tangent splash of $\Bpi$.

\begin{lemma}\Label{defnBsplash}
The tangent splash $\ST$ for $\Bpi$ has centre $T=(1,0,0)$ and points $V_{d,f}=
\ell_{0,d,f}\cap\li=(1-d+f\tau,1,0)$ for $d,f\in\GF(q)$.
 Alternatively, we can write $\ST=\{(a+b\tau,1,0)\st a,b\in\GF(q)\}\cup\{(1,0,0)\}$.
\end{lemma}

\begin{proof}
 By Lemma~\ref{allextlines},
 the lines of $\Bpi$ are $\ell_{0,d,f}$. 
Let $V_{d,f}= \ell_{0,d,f}\cap\li=(1-d+f\tau,1,0)$, 
so the splash of $\Bpi$ is
$\ST=\{T\}\cup\{V_{d,f}\st f,d\in\GF(q)\}$. It is straightforward to rewrite
this as $\{(a+b\tau,1,0)\st a,b\in\GF(q)\}\cup\{(1,0,0)\}$.
\end{proof}

\section{Linear Sets}\Label{sec:new-lin-set}

Linear sets have many applications to projective geometry, and recent work by Lavrauw and Van de Voorde \cite{lavr10}  studies them in great detail. 
We will show that a tangent splash of the line $\li\cong\PG(1,q^3)$ is a $\GF(q)$-linear set of rank 3.
We begin by defining linear sets of $\PG(1,q^3)$ (more generally, linear sets of $\PG(n-1,q^t)$ are defined in \cite{lavr10}). 
The points of $\PG(1,q^3)$ can be considered as elements of   the 2-dimensional vector space $V=\GF(q^3)^2$ over 
$\GF(q^3)$.
 Let $U$ be a subset of $V$ that forms a 3-dimensional vector space over $\GF(q)$. Then the vectors of $U$ (considered as vectors over $\GF(q^3)$) form a $\GF(q)$-linear set of rank 3 of $\PG(1,q^3)$.

\begin{theorem}\Label{splash-is-linear}
Let $\ST$ be a tangent splash of $\li\cong\PG(1,q^3)$ in $\PG(2,q^3)$. Then $\ST$ is a $\GF(q)$-linear set of rank 3 and size $q^2+1$. Conversely, a $\GF(q)$-linear set of rank 3 and size $q^2+1$ is a tangent splash.
\end{theorem}

\begin{proof}
By Theorem~\ref{transitiveontangentsubplanes}, we can without loss of generality prove this for the tangent splash of the \orsp\ $\Bpi$ coordinatised in Section~\ref{coord-subplane}. By Lemma~\ref{defnBsplash}, $\Bpi$ has
 tangent splash $\ST=\{(a+b\tau,1,0)\st a,b\in\GF(q)\}\cup\{(1,0,0)\}$ (with centre $T=(1,0,0)$). That is, $\ST$ is the span over $\GF(q)$ of the three linearly independent vectors 
$(1,0,0),(0,1,0)$ and $(\tau,1,0)$. Hence a tangent splash of $\PG(1,q^3)$ is a $\GF(q)$-linear set of rank 3, and size $q^2+1$. 
Finally, \cite[Theorem 5]{lavr10} shows that all $\GF(q)$-linear sets of rank 3, and size $q^2+1$ are projectively equivalent.
\end{proof}

We can interpret $\GF(q)$-linear sets of $\PG(1,q^3)$ of rank 3
in the Bruck-Bose representation in $\PG(5,q)$ as described in \cite{lavr10}. Recall that in this representation, a point $A$ of $\PG(1,q^3)$ corresponds to a plane $[A]$ of a regular 2-spread $\S$ in $\PG(5,q)$. 
Let $\Xl$ be  a  tangent splash of $\li\cong\PG(1,q^3)$, so $\Xl$ is a $\GF(q)$-linear set of rank 3 of size $q^2+1$. Let $[\Xl]$ be the corresponding set of planes of the regular 2-spread $\S$ in $\PG(5,q)$. Then as $\Xl$ is a linear set, there is a subspace $\Pi$ of $\PG(5,q)$, such that the planes of $[\Xl]$ are exactly the planes of $\S$ that meet $\Pi$. 
Clearly, $\Pi$ cannot be a 4-space. Further, by \cite[Lemma 10]{lavr10}, each 3-space of $\PG(5,q)$ meets every plane of the regular 2-spread $\S$, so $\Pi$ cannot be a 3-space. Hence $\Pi$ is a plane, and so $\Pi$ meets one plane $[A]$ of $[\Xl]$ in a line, and the remaining $q^2$ elements of $[\Xl]$ in a point. So, using the terminology introduced in \cite{fanc}, and used in \cite{lavr10}, $\Xl$ is a {\em club} with {\em head} the point $A$ of $\Xl$. 
Moreover, it follows from the theory on linear sets that $A$ is the centre of the tangent splash $\Xl$, and there are $q^2+q+1$ planes that have this property.

\begin{corollary}\Label{coverplanes}
Let $\ST$ be a tangent splash of $\li$ with centre $T$, and let $[\ST]$ be
the corresponding set of planes in $\si$ in the Bruck-Bose representation in
 $\PG(6,q)$. 
There
are exactly $q^2+q+1$ planes of $\si\cong\PG(5,q)$ that meet every plane of
$[\ST]$, called {\bf cover planes}. The cover planes
each meet the centre $[T]$ in distinct lines, and meet every other plane of $[\ST]$ in 
distinct points.
\end{corollary}

We can interpret results known about linear sets to get some further interesting results about tangent splashes. In Theorem~\ref{nos5-version2}, we show that two tangent splashes with a common centre meet in at most $q+1$ points. 
In \cite{ferr03}, it is  shown that two tangent splashes meet in at most $2q+2$ points. We have run some computer simulations which suggest that this bound can be improved to $2q$. We conject that if two tangent splashes 
${\mathscr S}_{T_1}$, ${\mathscr S}_{T_2}$, meet in a set $X$ of $2q$ points, then $X$ consists of two \orsls\ that meet in the points $T_1$ and $T_2$.

The following result from \cite{lavr10} is particularly interesting to us, note that we state it only for the particular case when a linear set is a tangent splash. It allows us to prove a useful corollary about \orsls\ in a tangent splash.

\begin{theorem}
{\rm \cite[Corollary 13]{lavr10}}
\Label{cor13}
Through two non-centre points of a tangent splash $\ST$, there is exactly one \orsl\ contained in $\ST$, further, it contains the centre.
\end{theorem}

\begin{corollary}\Label{nos6}
Let $\ST$ be a tangent splash of $\PG(2,q^3)$. Then every \orsl\ contained in $\ST$ contains $T$.  
\end{corollary}

We note that this also allows us to demonstrate a bijection. 
If $\pi$ is a tangent \orsp\ with tangent splash $\ST$, then the \orsls\ of $\pi$ through a point $P\in\pi\setminus\{T\}$ meet $\ST$ in an \orsl\ through $T$. Hence each \orsl\ of $\ST$ corresponds to a unique point of $\pi\setminus\{T\}$, that is:

\begin{corollary}\Label{cor:orsl-corr-to}
Let $\pi$ be a tangent \orsp\ with tangent splash $\ST$ on $\li$. The $q^2+q$ \orsls\ of $\ST$ are in one to one correspondence with the points of $\pi\setminus\{T\}$.
\end{corollary}

We complete this section with a short discussion about the different constructions of tangent splashes and linear sets from \orsps\ of $\PG(2,q^3)$. 
In \cite{luna04}, it is shown that every linear set is a projection of a canonical geometry. That is, a tangent splash $\ST$ on $\li$ can be obtained by taking an \orsp\ $\pi$ in $\PG(2,q^3)$, and a point $P\notin\pi$, and projecting $\pi$ from $P$ to $\li$. We note that this is a different geometric construction to ours, which constructs the tangent splash of $\pi$ as the intersection of the lines of $\pi$ with $\li$. Given an \orsp\ $\pi$ with tangent splash $\ST$ on $\li$, we conject that there is no point $P$ that projects $\pi$ onto $\ST$; in order to get a projection that is $\ST$, we conject that we need to project an \orsp\ $\pi'$ exterior to $\li$ from a point $P\notin\pi'$ that lies on a line of $\pi'$.

\section{Tangent splashes and Sherk surfaces}\Label{sect:sherk}

In this section we show that a tangent splash is a Sherk surface.
We can uniquely identify points $(a_0,a_1,a_2,1)$ of $\AG(3,q)$,
$a_0,a_1,a_2\in\GF(q)$,  with
elements $a_0+a_1\tau+a_2\tau^2$ of $\GF(q^3)$. 
In \cite{sher86}, Sherk considers this representation of $\A=\AG(3,q)$ and lets
$\overline\A=\AG(3,q)\cup\{\infty\}$. 
A {\em Sherk surface} ${\mathscr S}(\theta,K,L,\phi)$, $\theta,\phi\in\GF(q)$,
  $K,L\in\GF(q^3)$, is the set of points $X$ in $\overline \A$ satisfying
$$\theta X^{1+q+q^2}+(KX^{1+q}+K^qX^{q+q^2}+K^{q^2}X^{1+q^2})+(LX+L^qX^q+L^{q^2}X^{q^2})+\phi=0.$$
There are four orbits of Sherk surfaces, uniquely determined by the
  size of the surfaces in each orbit. A Sherk surface has size 1, $q^2-q+1,\
  q^2+1,$ or $q^2+q+1$.
The Sherk surfaces of size $q^2+q+1$ are precisely the Bruck
  covers/hyper-reguli, see \cite{john07}, 
  and are equivalent to the exterior splash of an exterior \orsp\ of $\PG(2,q^3)$, see \cite{BM-ext}. 
  We show here that a tangent splash of $\PG(2,q^3)$ is a Sherk
surface of size $q^2+1$.

\begin{theorem} A tangent splash of $\PG(2,q^3)$ is a Sherk surface of size $q^2+1$.
\end{theorem}

\begin{proof} We can naturally identify the points of $\li=\PG(1,q^3)$ with elements
  of $\GF(q^3)\cup\{\infty\}$: the point $(a,1,0)\in\li$ corresponds
  to the element $a\in\GF(q^3)$, and the point $(1,0,0)$ corresponds to the
  element $\infty$. Theorem~\ref{transitiveontangentsubplanes} and
  Lemma~\ref{defnBsplash}, we can 
without loss of generality,  take our tangent splash to
  be $\ST=\{(a+b\tau,1,0)\st a,b\in\GF(q)\}\cup\{(1,0,0)\}$. 
The four points $(1,0,0), (0,1,0),(1,1,0), (\tau,1,0)$ are contained in
$\ST$, and are not on an \orsl,  so by Theorem~\ref{nos5-version2-part2}, they uniquely determine $\ST$.

Sherk shows that a plane of
  $\A=\AG(3,q)$ together with $\infty$ is a Sherk surface ${\mathscr
    S}(0,0,L,\phi)$, $L\in\GF(q^3)$, $\phi\in\GF(q)$ of size $q^2+1$. 
This Sherk surface is the set of points satisfying
$$(LX+L^qX^q+L^{q^2}X^{q^2})+\phi=0.$$
We
  show that our tangent splash is one of these planes. 
The  centre of our tangent splash $\ST$ is the point $T=(1,0,0)$, which corresponds to the element $\infty$, which is
on every  Sherk surface that corresponds to a plane in $\AG(3,q)$. The point $(0,1,0)$ corresponds to the
element $0$ of $\GF(q^3)$. This is on the Sherk surface ${\mathscr
    S}(0,0,L,\phi)$ if and only if $\phi=0$. Now consider the points
  $(1,1,0), (\tau,1,0)$ of $\ST$ which correspond to the elements $1,\tau$ of
  $\GF(q^3)$ respectively. These lie on the Sherk surface ${\mathscr
    S}(0,0,L,0)$ if
\begin{eqnarray}
L+L^q+L^{q^2}&=&0,\label{sherk1}\\
L\tau+L^q\tau^q+L^{q^2}\tau^{q^2}&=&0 \label{sherk2}.
\end{eqnarray}
These equations have a unique nonzero solution for $L$, since in $\AG(3,q)$
there is a unique plane through the three points $0,1,\tau$, hence there is a
unique Sherk surface ${\mathscr
    S}(0,0,L,\phi)$ containing these three points. Note this Sherk surface also
  contains $\infty$.

Now $a\times $(\ref{sherk1})$+b\times$(\ref{sherk2}) gives: 
$L(a+b\tau)+L^q(a+b\tau^q)+L^{q^2}(a+b\tau^{q^2})=0.$
If $a,b\in\GF(q)$, then this can be written as
$$L(a+b\tau)+L^q(a+b\tau)^q+L^{q^2}(a+b\tau)^{q^2}=0.$$ Hence the element
$a+b\tau$ of $\GF(q^3)$ 
lies on the Sherk surface ${\mathscr
    S}(0,0,L,0)$. That is, the Sherk
surface containing the elements $\infty,0,1,\tau$ contains all the elements
$a+b\tau$, where
$a,b\in\GF(q)$. So this Sherk surface is equivalent to our tangent splash
$\ST=\{(a+b\tau,1,0)\st a,b\in\GF(q)\}\cup\{(1,0,0)\}$. 
\end{proof}


\section{The tangent subspace at an affine point}\Label{sect:tgt-subspace}

In this section we show that each point $P$ in a tangent  \orsp\ $\pi$ has a natural correspondence to a plane $[P^\perp]$ in $\PG(6,q)$.
Consider an \orsp\ $\Bpi$ in $\PG(2,\r^3)$ tangent to $\li$ at the point $T$.  For
each affine point $P$ of $\Bpi$, 
we construct an \orsp\ $P^\perp$ that contains $P$ and is secant to $\li$ as follows. Let
$\ell_1,\ldots,\ell_{q+1}$ be the $q+1$ lines of $\Bpi$ through
$P$. Then by Lemma~\ref{numberofsublines}(\ref{sublineprojection}),
$m=\{\ell_i\cap\li\st i=1,\ldots,q+1\}$ is an \orsl\ of $\li$
through $T$. Now
$m$ and $PT\cap\Bpi$ are two \orsls\  through $T$, and so lie in a unique
\orsp\ which we denote by $P^\perp$. We show in
Corollary~\ref{tgtsubspacecor} that the \orsps\ $\Bpi$
and $P^\perp$ meet in exactly the \orsl\ $PT\cap\Bpi$.

By Theorem~\ref{FFA-subplane-tangent}, in $\PG(6,q)$, $\Bpi$
corresponds to a ruled surface $\bbb$ and $P^\perp$ corresponds to  a plane $[P^\perp]$ that meets 
$q+1$ elements of the 2-spread $\S$. We will show in
Theorem~\ref{BBPperp} that
$[P^\perp]$ is the tangent space to $\bbb$ at the point $P$. 
This leads us to call $P^\perp$ the {\em tangent space} of $P$.

By Theorem~\ref{transitiveontangentsubplanes} we can without loss of
generality 
let $\Bpi$ be the \orsp\ coordinatised in
Section~\ref{coord-subplane}. 
Consider the point $P_{e,d}$ of $\Bpi$ for some $e,d\in\GF(q)$ and consider
the \orsp\ $P_{e,d}^\perp$. 
We work in $\PG(6,q)$ and in the next lemma determine the coordinates of two points in  $[P_{e,d}^\perp]\cap\si$.
We use the following notation for $e\in\GF(q)$:
\begin{eqnarray*}
\minustheta(e)&=&(e+\tau)^{\r^2+\r}=(e+\tau^q)(e+\tau^{q^2}),\\ 
\plustheta(e)&=&(e+\tau)^{\r^2+\r+1}=(e+\tau)(e+\tau^q)(e+\tau^{q^2}).
\end{eqnarray*}
Note that since $\plustheta(e)^\r=\plustheta(e)$, it follows that
$\plustheta(e)\in\GF(\r)$. We again use the generalised Bruck-Bose map
$\sigma$ defined in Section~\ref{BBintro}.

\begin{lemma}\Label{BBPperp0}
 In $\PG(6,q)$, the plane $[P_{e,d}^\perp]$, $d,e\in\GF(q)$, contains the
 points $J_{e,d}$, $K_{e,d}$ of $\si$, where
$$J_{e,d}=\sigma((1-d)\tau\minustheta(e)^2,\allowbreak\tau\minustheta(e)^2,0),
\quad
K_{e,d}=\sigma((1-d+e+\tau)\tau\minustheta(e)^2,\tau\minustheta(e)^2,0).$$
\end{lemma}

\begin{proof}
In $\PG(2,q^3)$, the \orsp\ $P_{e,d}^\perp$ contains all the affine points
of the \orsl\ $m_e=TP_{e,d}\cap\Bpi$. 
Firstly, suppose that $d\ne 0$, so $P_{e,d}$ and $P_{e,0}$ are distinct
points of $m_e$.  Consider the two lines $\ell_{e,d,0}$ and $\ell_{e,d,1}$ through
$P_{e,d}$.
Using Table~\ref{table-coords} we calculate
$V_{e,d,0}=\ell_{e,d,0}\cap\li=(-d+1,1,0)$ and 
$V_{e,d,1}=\ell_{e,d,1}\cap\li=(e-d+1+\tau,1,0)$.
Now $P_{e,d},P_{e,0},V_{e,d,0},V_{e,d,1}$ are all distinct points of $P_{e,d}^\perp$,
so $X_{e,d}=P_{e,d}V_{e,d,1}\cup P_{e,0}V_{e,d,0}$ is in $P_{e,d}^\perp$. 
Straightforward calculation shows
$X_{e,d}=(e^2+(e-d+d^2)\tau,e^2+(e-d)\tau,(e+\tau)^2)$.
Note that the lines
$P_{e,0}X_{e,d}$, $P_{e,d}X_{e,d}$ meet
$\li$ in points of the \orsp\ $P_{e,d}^\perp$.
The \orsl\ of $P_{e,d}^\perp$ through $P_{e,d}$ and $X_{e,d}$ contains the 
point $V_{e,d,1}$ of $\li$, so by Theorem~\ref{sublinesinBB}, corresponds in $\PG(6,q)$ to the line
through $P_{e,d}$ and $X_{e,d}$. Hence in $\PG(6,q)$, the point
$P_{e,d}X_{e,d}\cap\si$ lies in the plane $[P_{e,d}^\perp]$. Similarly the
point $P_{e,0}X_{e,d}\cap\si$ lies in $[P_{e,d}^\perp]$. 
To calculate the coordinates of these points in $\PG(6,q)$, we need to take
the coordinates of 
 $X_{e,d},P_{e,0},P_{e,d}$ in $\PG(2,q^3)$ and write them
with third coordinate in $\GF(q)$. So using the Bruck-Bose map $\sigma$, we
have in $\PG(6,q)$,
\begin{eqnarray*}
X_{e,d}&=&\sigma((e^2+(e-d+d^2)\tau)\minustheta(e)^2,(e^2+(e-d)\tau)\minustheta(e)^2,\plustheta(e)^2),\\
P_{e,0}&=&\sigma(e(e+\tau)\minustheta(e)^2,e(e+\tau)\minustheta(e)^2,\plustheta(e)^2),\\
P_{e,d}&=&\sigma((e+d\tau)(e+\tau)\minustheta(e)^2,e(e+\tau)\minustheta(e)^2,\plustheta(e)^2).
\end{eqnarray*}
Now $X_{e,d}=P_{e,0}-dJ_{e,d}$,
hence the line $P_{e,0}X_{e,d}$ of $\PG(6,q)$ meets $\si$ in the point
$J_{e,d}$, so $J_{e,d}$ is in the plane $[P_{e,d}^\perp]$.
Similarly, $X_{e,d}= P_{e,d}-dK_{e,d}$,
hence the line $P_{e,d}X_{e,d}$ of $\PG(6,q)$ meets $\si$ in the point
$K_{e,d}$, so $K_{e,d}$ is in the plane $[P_{e,d}^\perp]$.

Now suppose $d=0$. In this case, let $X_e=P_{e,1}V_{e,0,1}\cap
P_{e,0}V_{e,0,0}$.
Straightforward calculation gives
\begin{eqnarray*}
X_e&=&(e^2+e\tau-\tau,e^2+e\tau-\tau,(e+\tau)^2)\\
&=&((e^2+e\tau-\tau)\plustheta^{-}(e)^2,(e^2+e\tau-\tau)\plustheta^{-}(e)^2,\plustheta(e)^2).
\end{eqnarray*}
Then in $\PG(6,q)$,  $X_e=
P_{e,0}- J_{e,0}$ and $X_e=
P_{e,1}- K_{e,0}$. So $J_{e,0},K_{e,0}\in
[P_{e,d}^\perp]$
as required. 
\end{proof}

\begin{theorem}\Label{BBPperp}
Let $\Bpi$ be a tangent \orsp\ in $\PG(2,q^3)$ and 
let $\bbb$ be the corresponding ruled surface in $\PG(6,q)$. By
Theorem~{\rm \ref{FFA-subplane-tangent}}, 
$\bbb$ is the intersection of nine quadrics.  Let $P$ be an affine point of
$\Bpi$, then in
$\PG(6,q)$, the plane $[P^\perp]$ is the intersection of the tangent
spaces at $P$ to each of these nine quadrics. 
\end{theorem}

\begin{proof} As before, without loss of generality we prove this for the
  tangent \orsp\ $\Bpi$ coordinatised in Section~\ref{coord-subplane}. 
The equations of the nine quadrics determining $\bbb$ are given in
\cite{barw12} in equations (16), (17), (18), and are: 
\begin{eqnarray}
(1-x)^\r((1+\w)y-1)-(1-x)((1+\w)y-1)^\r&=&0,\label{eqn2forsigma1}\\
(1-x)(\w y)^\r-(1-x)^\r\w y=0,\label{eqn3forsigma1}\\
(1-y)^\r(1-x)-(1-y)(1-x)^\r=0.\label{eqn1forsigma2}
\end{eqnarray}
These equations are in $\PG(2,q^3)$, and each equation corresponds to three
quadrics in $\PG(6,q)$. A point $(x,y,1)$ of
$\PG(2,q^3)$ 
satisfying one of these equations corresponds to a point $\sigma(x,y,1)$ in
$\PG(6,q)$
lying on all three quadrics. 
Note that equation (\ref{eqn2forsigma1}) is equation (\ref{eqn1forsigma2})
minus equation
(\ref{eqn3forsigma1}), so we only need consider the second two equations. 

By Lemma~\ref{Gextsubgp}, without loss of generality we can
prove the result for the affine  point $P_{e,d}$ in $\Bpi$ for some $e,d\in\GF(q)$. 
In $\PG(6,q)$ we want to find all lines through $P_{e,d}$ tangent to each
of the nine quadrics. 
A line $\ell$ through the point $P_{e,d}$ meets $\si$ in a point $R=\sigma(u,v,0)$
where $u,v\in\GF(q^3)$ depend on $e,d$. A general point on the line $\ell$ has form
$P_{e,d}+tR$, where $t\in\GF(\r)$.  To determine which lines $\ell$ are tangent
lines, we substitute this general point into
the equations for the quadrics and solve for points $R$ where $t=0$ is a double
root. 

From Table~\ref{table-coords} we have $P_{e,d}=(e+d\tau,e,e+\tau)$.
Writing this as $(x,y,1)$ gives
\begin{equation}\label{coordsxy}
x=(e+d\tau)\ee, \ y= e\ee,\ 
1-x=(1-d)\tau\ee,\ 1-y=\tau\ee.
\end{equation}
Replacing $x$ with $x+tu$ and $y$ with $y+tv$ in (\ref{eqn3forsigma1}) we obtain
\begin{eqnarray*}
0&=&(1-x-tu)(\tau(y+tv))^\r-(1-x-tu)^\r\tau(y+tv)\\
&=&(1-x)(\w y)^\r-(1-x)^\r\w y+t^2(-u\tau^\r v^\r+u^\r\tau v)+\\
&&\quad\quad +t(-u\tau^\r y^\r+(1-x)\tau^\r v^\r-(1-x)^\r\tau v+u^\r\tau y).
\end{eqnarray*}
Note that $P_{e,d}=(x,y,1)$ satisfies (\ref{eqn3forsigma1}), so when the
above equation is regarded as a polynomial in $t$, the constant term is
equal to 0.  For $t=0$ to be a repeated root, we need the coefficient of
$t$ to be equal to zero, that is $-u\tau^\r y^\r+(1-x)\tau^\r v^\r-(1-x)^\r\tau v+u^\r\tau y=0$.
Substituting for $x,y$ using (\ref{coordsxy}) yields
\begin{eqnarray*}
0&=&-u\tau^\r e\eeq+(1-d)\tau\ee\tau^\r v^\r-(1-d)\tau^\r\eeq\tau v+u^\r\tau e\ee.
\end{eqnarray*}
Rearranging gives
\begin{eqnarray*}
\frac 1{\tau\minustheta(e)}\left((1-d)\tau v+eu\right)&=&
\left(\frac 1{\tau\minustheta(e)}\left((1-d)\tau v+eu\right)\right)^\r,\\
\end{eqnarray*}
and so $\left((1-d)\tau v+eu\right)/\tau\minustheta(e)$ is in
$\GF(\r)$. So  we want  $u,v\in\GF(q^3)$ such that $\left((1-d)\tau
  v+eu\right)/\tau\minustheta(e)=a$, for any $a\in\GF(q)$. That is, 
\begin{eqnarray}
 (1-d)\tau v+eu&=&a\tau\minustheta(e)\label{firstped}
\end{eqnarray}
for $a\in\GF(\r)$.
We now repeat these calculations for
equation (\ref{eqn1forsigma2}): we replace $x$ by $x+tu$ and $y$ by $y+tv$ to obtain
\[
(1-y-tv)^\r(1-x-tu)-(1-y-tv)(1-x-tu)^\r=0,
\]
setting the coefficient of $t$  to 0 gives
\[
-v^\r(1-x)-(1-y)^\r u+(1-y)u^\r+v(1-x)^\r=0,
\]
and substituting for $x,y$ from (\ref{coordsxy}) gives $((1-d)v-u)/\tau
\minustheta(e)$ is an element of $\GF(q)$. So we want $u,v\in\GF(q^3)$ such that 
\begin{eqnarray}
(1-d)v-u=b\tau \minustheta(e)\label{secondped}
\end{eqnarray}
for any $b\in\GF(q)$. 
Solving (\ref{firstped}) and (\ref{secondped}) for $u,v$ gives
$$u=\frac 1{\plustheta(e)}(a+be)\tau\minustheta(e)^2-b\tau\minustheta(e),\quad v=\frac
1{1-d}\frac 1{\plustheta(e)}(a+be)\tau\minustheta(e)^2,$$
where $a,b$ are any values in $\GF(q)$, that is, $u,v$ are functions of $a,b\in\GF(q)$.
  Hence the points $R$ for which $P_{e,d}R$ is a tangent line to the nine quadrics
forming the variety $\bbb$ are, for any $a,b\in\GF(q)$, 
\begin{eqnarray*}
R&=&\sigma(u,v,0)\\
&=& \left(\frac{a+be}{(1-d)\plustheta(e)}\right) \sigma\left((1-d)\tau\minustheta(e)^2,\tau\minustheta(e)^2,0\right)-b\sigma(\tau\minustheta(e),0,0).
\end{eqnarray*}
Note that $(a+be)/((1-d)\plustheta(e))\in\GF(q)$. By varying
$a,b\in\GF(q)$, we see that $R$ is any point on the line
joining $J_{e,d}=\sigma((1-d)\tau\minustheta(e)^2,\tau\minustheta(e)^2,0)$ and $C_e=\sigma(\tau\minustheta(e),0,0)$.
Hence the lines through $P_{e,d}$ that are tangent to $\bbb$ form a plane
$\alpha=\langle P_{e,d},J_{e,d},C_e\rangle$. 
By Lemma~\ref{BBPperp0}, the point
$J_{e,d}$ is in the plane $[P_{e,d}^\perp]$.
Further, the point $K_{e,d}=J_{e,d}+\plustheta(e)C_e$ is in $\alpha$
as $\plustheta(e)\in\GF(q)$, and is in $[P_{e,d}^\perp]$ by Lemma~\ref{BBPperp0}.
Hence $\alpha=[P_{e,d}^\perp]$, that is, the plane $[P_{e,d}^\perp]$ is formed from the lines tangent to $\bbb$ at $P_{e,d}$.
\end{proof}

\begin{corollary}\Label{tgtsubspacecor}
Let $\Bpi$ be an \orsp\ of $\PG(2,q^3)$ tangent to $\li$ at the point $T$, and let $P$ be an
affine point of $\Bpi$. Then the \orsps\
$\Bpi$ and $P^\perp$ meet in exactly the \orsl\
$PT\cap\Bpi$. Further, there is a bijection from affine points $P$ of $\Bpi$ to
\orsps\ $P^\perp$ that contain the \orsl\ $PT\cap\Bpi$ and meet $\li$ in an \orsl\ contained in $\ST$. 
\end{corollary}

\begin{proof} The proof of Theorem~\ref{BBPperp} shows that $\Bpi$ and
  $P^\perp$ meet in the \orsl\ $PT\cap\Bpi$. The bijection follows from
  Lemma~\ref{numberofsublines}(\ref{tgtsplash2pts}).
\end{proof}

The next corollary states a result from the final paragraph of the proof of Theorem~\ref{BBPperp}. This result is needed in another article. 

\begin{corollary} $[P_{e,d}^\perp]=\langle P_{e,d},J_{e,d},C_e\rangle$. 
\end{corollary}


\section{Conclusion}

This article explored further properties of a tangent \orsp\ $\Bpi$ in
$\PG(2,\r^3)$.  In particular, we investigated the \emph{tangent splash} of $\Bpi$
(the intersection of the lines of $\Bpi$ with $\li$).
Most of this investigation was carried out in the plane
$\PG(2,q^3)$. However, the tangent splash has interesting properties when
looked at in the Bruck-Bose representation of $\PG(2,q^3)$ in
$\PG(6,q)$. In particular, the cover planes are of interest, and we
investigated the tangent subspace of an affine point of $\Bpi$. 
Moreover, the relationship between tangent splashes and linear sets is particularly interesting, and known results about linear sets can be interpreted in this setting.

In another article \cite{barwtgt2}, the authors continue this investigation
by working in $\PG(6,q)$ and looking at the interaction between the ruled
surface corresponding to $\Bpi$ and the set of planes forming the tangent
splash of $\Bpi$.

{\bf Acknowledgment} We would like to thank the anonymous referees. One for pointing out the important relationship between tangent splashes and linear sets, and that a number of our results could be proved more directly using the theory of linear sets. The other for helpful suggestions regarding Section 4, and for noting that the affine plane of Theorem~\ref{splash-is-affine-plane} is Desarguesian.

\end{document}